\documentclass[11pt]{article}
\usepackage[a4paper,margin=1in]{geometry}
\usepackage{amsmath,amssymb,amsthm,mathtools}
\usepackage{enumitem}
\usepackage{hyperref}
\hypersetup{
  colorlinks=true,
  linkcolor=blue,
  citecolor=blue,
  urlcolor=blue
}
\newtheorem{theorem}{Theorem}[section]
\newtheorem{lemma}[theorem]{Lemma}
\newtheorem{proposition}[theorem]{Proposition}

\newcommand{\C}{\mathbb C}
\newcommand{\R}{\mathbb R}
\newcommand{\Dc}{D_{\mathbb C}^2}

\newcommand{\SL}{\mathfrak{sl}}
\newcommand{\tr}{\operatorname{tr}}
\newcommand{\norm}[1]{\left\lVert #1\right\rVert}
\title{\(C^{1,1}\) Regularity Global Estimates for Homogeneous Complex Hessian Equations on Punctured Domains}
\author{Zhenghuan Gao\textsuperscript{1}, Xi-Nan Ma\textsuperscript{2}, Bao Yu\textsuperscript{2}, and Dekai Zhang\textsuperscript{3}}
\date{}
\begin{document}
\maketitle
\begin{abstract}
In this paper, we prove the \(C^{1,1}\) regularity of Green functions
associated with the complex \(k\)-Hessian operator for \(1\leq k<n\), and give
a new proof of the corresponding regularity of pluricomplex Green functions.
We establish real Hessian estimates for approximating problems associated
with homogeneous complex Hessian equations on punctured domains. The main idea
is to introduce new auxiliary functions generated by complex linear
vector fields to reduce the global real Hessian estimate to
the boundary. We also treat the complex
Monge-Amp\`ere case. Similar real Hessian estimates also works for the exterior problem.

\end{abstract}

\medskip


\section{Introduction}
Let \(\Omega\) be a domain in \(\mathbb C^n\) and \(0\in\Omega\). Let \(u\) be a solution to the following homogeneous complex \(k\)-Hessian equation,
\begin{equation}\label{maineq::m-subharmonic-Green-function}
    \begin{cases}
        (dd^cu)^k\wedge \omega^{n-k}=0&\quad\text{in }\Omega\setminus\{0\},\\
        u=-1&\quad\text{on }\partial\Omega,\\
        u=-|z|^{2-\frac{2n}k}+O(1)&\quad\text{as }z\rightarrow 0.
    \end{cases}
\end{equation}
Here \(dd^c=2\sqrt{-1}\partial\bar\partial\) and \(\omega=dd^c|z|^2\). 
We also call $u$ as the Green function associated with the complex $k$-Hessian operator.
The main goal of this paper is to prove the following result.

\begin{theorem}\label{thm:green-k}
    Let \(\Omega\subset\mathbb C^n\) be a smooth bounded (\(k\)-\(1\))-pseudoconvex domain containing the origin, \(1\le k<n\). Let \(u\) be the \(k\)-subharmonic solution of \eqref{maineq::m-subharmonic-Green-function}. Then \(u\in C_{\text{loc}}^{1,1}(\overline\Omega\setminus\{0\})\) and  for a.e. $z\in \overline\Omega\setminus\{0\}$, 
    \begin{align*}
|D^2 u(z)|\le C|z|^{-\frac{2n}{k}}.   
    \end{align*}
\end{theorem}

The study of existence of pluricomplex Green function on pseudoconvex domain goes back to the work of B. Guan in \cite{Guan1998,Guan2007} and B\l{}ocki in \cite{Blocki2000}.
For smooth bounded strongly pseudoconvex domains,
B. Guan \cite{Guan1998}
constructed smooth approximating solutions and established
\(C^{1,\alpha}\) regularity of the pluricomplex Green function which is a solution to the following problem  
\[\begin{cases}
    (dd^cu)^n=0&\quad\text{in }\Omega\setminus\{0\},\\
    u=0&\quad\text{on }\partial\Omega,\\
    u=\log|z|+O(1)&\quad\text{as }z\rightarrow 0.
\end{cases}\]  
Moreover, he proved the following important  gradient estimate and complex Hessian estimates
\begin{align*}
|Du|^2(z)+\Delta u(z)\le C|z|^{-2}.
\end{align*}
Using holomorphic changes of the variables
and second differences, B\l{}ocki \cite{Blocki2000} proved the \(C^{1,1}\) regularity . 

These results motivate the regularity problem for the more general complex
\(k\)-Hessian operator for \(1\leq k<n\), whose fundamental solution has a
power singularity rather than a logarithmic singularity.  
The authors in \cite{GaoMaZhang2023} proved the existence and
uniqueness of a \(C^{1,\alpha}\) solution to
\eqref{maineq::m-subharmonic-Green-function}. They established the estimates
\[
  |Du(z)|\leq C|z|^{1-2n/k},
  \qquad \Delta u(z)\leq C|z|^{-2n/k}.
\]
In order to get above results, they constructed a suitable subsolution to solve the approximating \(k\)-Hessian equation in \(\Omega\setminus \overline{B_r}\). Then they proved uniform gradient and complex Hessian estimates for the
approximating solutions, as well as real Hessian estimates on the boundary.

What remained unresolved is a uniform
 real Hessian estimate. Indeed, control
of \(D_{\mathbb C}^2u=(u_{i\bar j})\) does not directly control the pure
second derivatives \(u_{ij}\), and a bounded Laplacian does not in general
yield a bounded real Hessian. Moreover, the degeneracy of the homogeneous
equation cannot allow us to apply the uniformly elliptic regularity
theory with constants independent of the approximation parameters.

Our aim is to establish the real Hessian estimate and thereby
improve the regularity from \(C^{1,\alpha}\) to \(C^{1,1}\). Throughout this
paper, \(Du\) and \(D^2u\) denote the real gradient and real Hessian, while
\[
  \Dc u=(u_{i\bar j})_{1\leq i,j\leq n},
  \qquad
  u_{i\bar j}=\frac{\partial^2u}{\partial z_i\partial\bar z_j},
\]
denotes the complex Hessian matrix. 

Let \(\Omega_r=\Omega\setminus\overline{B_r(0)}\). Let \(u^{\varepsilon,r}\) be the solution to the approximating problem
\begin{equation}\label{eq:intro-k-approx}
  \begin{cases}
    \sigma_k(\Dc u^{\varepsilon,r})=\varepsilon &\text{in }\Omega_r,\\
    u^{\varepsilon,r}=\underline u &\text{on }\partial B_r(0),\\
    u^{\varepsilon,r}=-1 &\text{on }\partial\Omega,
  \end{cases}
\end{equation}
where \(\underline u\) is a well-designed  subsolution. According to \cite{GaoMaZhang2023}, \(u^{\varepsilon,r}\) satisfies the following estimates,
\[
  C_0^{-1}|z|^\gamma\leq -u^{\varepsilon,r}(z)\leq C_0|z|^\gamma,
  \qquad |z|\,|Du^{\varepsilon,r}(z)|\leq C_0(-u^{\varepsilon,r}(z))
  \quad\text{in }\Omega_r,
\]
and
\[
  |z|^2|D^2u^{\varepsilon,r}(z)|\leq C_0(-u^{\varepsilon,r}(z))
  \quad\text{on }\partial\Omega_r,
\]
where \(C_0\) is a positive constant independent of \(\varepsilon\) and \(r\).
We prove the following estimate.

\begin{theorem}
\label{thm:intro-k-hessian}
Let \(1\leq k<n\), let \(\Omega\) be as in Theorem~\ref{thm:green-k}, and set
\(\gamma=2-2n/k\). Suppose that
\(u^{\varepsilon,r}\in C^4(\Omega_r)\cap C^2(\overline{\Omega_r})\)
is a negative \(k\)-admissible solution of \eqref{eq:intro-k-approx}.
Then there is a constant \(C\), depending only on \(n,k,\Omega\), such that
\begin{equation}\label{eq:intro-k-hessian}
  |D^2u^{\varepsilon,r}(z)|\leq C|z|^{-2n/k}
  \quad\text{in }\overline{\Omega_r}.
\end{equation}
In particular, \(C\) is independent of the approximation parameters
\(\varepsilon\) and \(r\).
\end{theorem}

The proof of the  real Hessian estimate for the approximating solution is based on the unitary invariance of the complex \(k\)-Hessian operator, the gradient estimate and the boundary Hessian estimate. We observe that for any point $z_0\in \Omega_r$ and in any  unit vector $\xi\in \mathbb C^n$, there exists a uniformly bounded matrix \(B=aI+K\), with \(K\) a skew Hermitian matrix satisfying $Bz_0=|z_0|\xi$. We then consider the function
$$
u^{\varepsilon,r}_t(z)=e^{-2at}u^{\varepsilon,r}(e^{tB}z),
$$
which preserves the complex \(k\)-Hessian equation. Differentiating the equation twice in \(t\) and using the concavity of \(\log\sigma_k\),  we prove that the function $$\frac{d^2u^{\varepsilon,r}_{t}}{dt^2}\bigg{|}_{t=0}+M u^{\varepsilon,r}$$ 
attains its maximum on the boundary. Together with the gradient estimates and the boundary Hessian estimate, we get the real Hessian estimate \eqref{eq:intro-k-hessian} and  this implies Theorem~\ref{thm:green-k}. The same argument also works for the exterior problem. One can see see Section 4 for details.

For the complex Hessian equations, Li \cite{Li2004} studied the smooth nondegenerate
Dirichlet problem, while B\l{}ocki \cite{Blocki2005} developed the weak solution
theory.
For the homogeneous complex Monge-Amp\`ere equation, B. Guan \cite{Guan2007}
proved \(C^{1,1}\) regularity of pluricomplex Green functions with pole
at infinity through the homogeneous complex Monge-Amp\`ere exterior
problem. Gao-Ma-Zhang's earlier work \cite{GaoMaZhangReal2023} obtained
weighted second derivative estimates for the homogeneous real
\(k\)-Hessian equation in punctured domains. Their work \cite{GaoMaZhangExterior} established \(C^{1,1}\)
regularity for the complex \(k\)-Hessian exterior problem. The present
paper improves the \(C^{1,\alpha}\) regularity in
\cite{GaoMaZhang2023} for the interior isolated-singularity problem
to \(C^{1,1}\).

The method also applies to the homogeneous complex Monge-Amp\`ere
equation. In this case, it gives a new proof of the \(C^{1,1}\) regularity
of the pluricomplex Green function with logarithmic pole established by B\l{}ocki
\cite{Blocki2000}.
The above argument does not work directly since $u\sim \log|z|$. For 
a trace free matrix \(B\in\mathfrak{sl}(n,\C)\) which is chosen so that \(Bz_0\) points in a prescribed real direction, we consider \(u_t(z)=u(e^{tB}z)\). Since \(\det e^{tB}=1\), the Monge--Ampère equation is preserved. Differentiating twice in \(t\) and using the concavity of \(\log\det\), the maximum principle reduces the global real Hessian estimate to the boundary Hessian and gradient estimates. For the exterior problem, we show $\frac{d^2u_{t}}{dt^2}\big{|}_{t=0}-C(\sum_{j}u_{z_j}z_j + \sum_{j}u_{\bar z_j}\bar z_j  )$ is unifomly bounded.

In addition, we  also obtain the real Hessian decay estimates of the solution to the exterior problem of homogeneous complex \(k\)-Hessian equation in \cite{GaoMaZhangExterior} and complex Monge-Amp\`ere equation in \cite{Guan2007}.

The paper is organized as follows. Section~\ref{sec:k-hessian} establishes
 the real Hessian estimate for the complex
\(k\)-Hessian equation, which yield
Theorem~\ref{thm:intro-k-hessian}.
Section~\ref{sec:cma-hessian} treats the real Hessian estimate for the  complex Monge-Amp\`ere equation
using traceless complex matrices. In Section \ref{section4}, we prove the real Hessian decay estimates for the exterior homogeneous complex \(k\)-Hessian equation and exterior complex Monge-Amp\`ere equation.

\section{The Real Hessian Estimate for the Complex Hessian Equation}\label{sec:k-hessian}
Throughout this section, let \(1\leq k<n\). We use the notation
\(
  \mathfrak u(n)=\{A\in M_n(\C): A^*=-A\}
\)
for the Lie algebra of skew-Hermitian matrices and \(A^*:=\bar A^T\) denotes the conjugate transpose of \(A\). For any Hermitian matrix $A$ and any vector $\xi$, we denote
\[A[\xi,\xi]=A_{i\bar j}\xi_{i}\bar \xi_{j}.\]

To get the Hessian estimate for the homogeneous Hessian equation, we need the following elementary lemma which says that any unit vector in \(\C^n\) can be
mapped to any other unit vector by a uniformly bounded operator in
\(\R I+\mathfrak u(n)\). For completeness, we include the proof.

\begin{lemma}\label{lem:linear-algebra}
For any \(u,v\in\C^n\) with \(\lvert u\rvert=\lvert v\rvert=1\), there exist \(a\in\R\) and \(K\in\mathfrak u(n)\) such that \(B:=aI+K\) satisfies \(Bu=v\). Moreover, there exists a constant \(C_n\), depending only on \(n\), such that
$$
  |a|+\norm{B}+\norm{B^2}\leq C_n.
$$

\end{lemma}

\begin{proof}
Suppose first that \(u=e_1\). Write \(v_1=a+ib\), where \(a,b\in\R\), and set
$$
K=
\begin{pmatrix}
ib & -\overline{v_2} & \cdots & -\overline{v_n}\\
v_2 & 0 & \cdots & 0\\
\vdots & \vdots & \ddots & \vdots\\
v_n & 0 & \cdots & 0
\end{pmatrix}.
$$
Then \(K^*=-K\) and \(Be_1=v\) for \(B=aI+K\). Since \(|v|=1\), we have \(|a|\leq1\) and
$$
\norm{K}^2
=b^2+2\sum_{j=2}^n|v_j|^2\leq2.
$$
Hence \(\norm{B}\leq1+\sqrt2\) and \(\norm{B^2}\leq\norm{B}^2\leq(1+\sqrt2)^2\).

For general \(u\), choose  the unitary matrix \(U\) such that \(Uu=e_1\), and put \(\widetilde v=Uv\). By the preceding construction, there exist \(a\in\R\) and \(\widetilde K\in\mathfrak u(n)\) such that \(\widetilde B:=aI+\widetilde K\) satisfies \(\widetilde B e_1=\widetilde v\), with the same bounds as above. Set \(K=U^{-1}\widetilde K U\) and \(B=aI+K=U^{-1}\widetilde B U\). Then \(K\in\mathfrak u(n)\), \(Bu=v\), and, since \(U\) is unitary, \(\norm{B}=\|\widetilde B\|\) and \(\norm{B^2}=\|\widetilde B^2\|\). The desired estimate follows.
\end{proof}

We recall the following estimates from \cite{GaoMaZhang2023}. In that work the
authors study the interior problem for homogeneous complex \(k\)-Hessian
equations through the approximating problem
\begin{equation}\label{eq:approx-problem}
\begin{cases}
  \sigma_k(\Dc u)=\varepsilon, & \text{in }\Omega_r,\\
  u=\underline u, & \text{on }\partial B_r,\\
  u=-1, & \text{on }\partial\Omega,
\end{cases}
\end{equation}
where \(\Omega_r=\Omega\setminus \overline{B_r(0)}\), and
\[
  \gamma=2-\frac{2n}{k}.
\]
If \(u\in C^4(\Omega_r)\cap C^2(\overline{\Omega_r})\) is a negative \(k\)-admissible solution of
\eqref{eq:approx-problem}, then one has the following uniform estimates
\begin{equation}\label{eq:known-estimates}
  c_0|z|^\gamma\leq -u(z)\leq C_0|z|^\gamma,\qquad
  |z|\,|Du(z)|\leq C_1(-u(z))\quad\text{in }\Omega_r,
\end{equation}
and
\begin{equation}\label{eq:boundary-second-estimate}
  |z|^2\,|D^2u(z)|\leq C_2(-u(z))
  \quad\text{on }\partial\Omega_r .
\end{equation}

The next proposition reduces the global real Hessian estimate to the boundary
estimate. 
Combining it with \eqref{eq:known-estimates} and
\eqref{eq:boundary-second-estimate} yields the global second derivative estimate.

\begin{proposition}\label{prop:global-hessian}
There exists a constant \(C\), independent of \(\varepsilon\) and \(r\), such
that
\[
  |z|^2\,|D^2u(z)|\leq C(-u(z))\qquad\text{in }\Omega_r .
\]
In particular,
\[
  |D^2u(z)|\leq C|z|^{-2n/k}.
\]
\end{proposition}
\begin{proof}
Fix \(z_0\in\Omega_r\) and a unit vector
\(\xi=(\xi_1,\ldots,\xi_{2n})\in\R^{2n}\). Set
\(\chi_j=\xi_j+\sqrt{-1}\xi_{j+n}\), \(1\leq j\leq n\). Then
\(|\chi|=1\). By Lemma~\ref{lem:linear-algebra}, there exist
\(a\in\R\) and \(K\in\mathfrak u(n)\) such that \(B:=aI+K\) satisfies
\[
  B\frac{z_0}{|z_0|}=\chi,
  \qquad
  |a|+\norm{B}+\norm{B^2}\leq C_n.
\]
Let \(E_t=e^{tB}\). Since \(B=aI+K\), we consider the small perturbation of the identity matrix
\(E_t=e^{at}U_t\), where \(U_t=e^{tK}\in U(n)\). For each fixed
\(z\in\Omega_r\) and \(t\) sufficiently close to zero, define
\[
  u_t(z)=e^{-2at}u(E_tz).
\]
Since \(E_t=e^{at}U_t\), a direct calculation gives
\[
  \Dc u_t(z)
  =
  U_t^T\Dc u(E_tz)\overline{U_t}.
\]
By the unitary invariance of \(\sigma_k\), it follows that
\[
  \sigma_k(\Dc u_t(z))
  =
  \sigma_k(\Dc u(E_tz))
  =
  \varepsilon.
\]

Let \(F=\log\sigma_k\). Then
\(F(\Dc u_t)=\log\varepsilon\). Differentiating with respect to \(t\)
gives
\begin{equation}\label{eq:hessian-first-variation}
  F^{i\bar j}(\Dc u_t)
  \left(\frac{du_t}{dt}\right)_{i\bar j}=0.
\end{equation}
Differentiating once more, we obtain
\begin{equation}\label{eq:hessian-second-variation}
\begin{aligned}
  0={}
  F^{i\bar j}(\Dc u_t)
  \left(\frac{d^2u_t}{dt^2}\right)_{i\bar j}
  +
  F^{i\bar j,p\bar q}(\Dc u_t)
  \left(\frac{du_t}{dt}\right)_{i\bar j}
  \left(\frac{du_t}{dt}\right)_{p\bar q}.
\end{aligned}
\end{equation}
Since \(F=\log\sigma_k\) is concave in \(\Gamma_k\), the second term in
\eqref{eq:hessian-second-variation} is nonpositive. Let \(t=0\) and we
 obtain
\[
  L\left(
    \left.\frac{d^2u_t}{dt^2}\right|_{t=0}
  \right)\geq0,
  \qquad
  L=F^{i\bar j}(\Dc u)\frac{\partial^2}
  {\partial z_i\partial\bar z_j}.
\]
By the homogeneity of \(\sigma_k\),
\[
  Lu=F^{i\bar j}(\Dc u)u_{i\bar j}=k.
\]

We next compute the second variation. Since
\(\frac{d}{dt}(E_tz)=BE_tz\), differentiating
\(u_t(z)=e^{-2at}u(E_tz)\) gives
\[
  \frac{du_t(z)}{dt}
  =
  e^{-2at}
  \bigl(Du(E_tz)\cdot BE_tz-2a\,u(E_tz)\bigr).
\]
Differentiating once more yields
\[
\begin{aligned}
  \frac{d^2u_t(z)}{dt^2}
  =e^{-2at}\bigl\{
  &D^2u(E_tz)[BE_tz,BE_tz]
   +Du(E_tz)\cdot B^2E_tz\\
  &-4a\,Du(E_tz)\cdot BE_tz
   +4a^2u(E_tz)
  \bigr\}.
\end{aligned}
\]
Thus, at \(t=0\),
\begin{equation}\label{eq:hessian-second-variation-explicit}
\left.\frac{d^2u_t(z)}{dt^2}\right|_{t=0}
=
D^2u(z)[Bz,Bz]+Du(z)\cdot B^2z
-4a\,Du(z)\cdot Bz+4a^2u(z).
\end{equation}

By \eqref{eq:known-estimates}, \eqref{eq:boundary-second-estimate}, and
Lemma~\ref{lem:linear-algebra}, we have on \(\partial\Omega_r\)
\[
  \left|
    \left.\frac{d^2u_t}{dt^2}\right|_{t=0}
  \right|
  \leq
  C\bigl(|z|^2|D^2u|+|z|\,|Du|+(-u)\bigr)
  \leq C(-u).
\]
Choose \(C\) sufficiently large and set
\[
  W=
  \left.\frac{d^2u_t}{dt^2}\right|_{t=0}+Cu.
\]
Then \(W\leq0\) on \(\partial\Omega_r\), while
\[
  LW
  =
  L\left(
    \left.\frac{d^2u_t}{dt^2}\right|_{t=0}
  \right)+CLu
  \geq Ck>0
\]
in \(\Omega_r\). The maximum principle gives \(W\leq0\) in
\(\Omega_r\). Hence
\begin{equation}\label{eq:hessian-second-variation-bound}
  \left.\frac{d^2u_t(z)}{dt^2}\right|_{t=0}
  \leq C(-u(z))
  \qquad\text{in }\Omega_r.
\end{equation}

At \(z_0\), the choice of \(B\) gives \(Bz_0=|z_0|\chi\). Under the
standard identification \(\C^n\simeq\R^{2n}\), the vector \(\chi\)
corresponds to \(\xi\). Therefore
\[
  D^2u(z_0)[Bz_0,Bz_0]
  =
  |z_0|^2D^2u(z_0)[\xi,\xi].
\]
Moreover, by \eqref{eq:known-estimates} and the bounds for \(a\), \(B\),
and \(B^2\),
\[
  \left|
  Du(z_0)\cdot B^2z_0
  -4a\,Du(z_0)\cdot Bz_0
  +4a^2u(z_0)
  \right|
  \leq C(-u(z_0)).
\]
Combining this with
\eqref{eq:hessian-second-variation-explicit} and
\eqref{eq:hessian-second-variation-bound}, we obtain
\[
  |z_0|^2D^2u(z_0)[\xi,\xi]
  \leq C(-u(z_0)).
\]
Since \(z_0\) and \(\xi\) are arbitrary, every eigenvalue of the real
Hessian \(D^2u(z_0)\) is bounded above by
\(C(-u(z_0))|z_0|^{-2}\).

Since \(u\) is \(k\)-admissible, \(\sigma_1(\Dc u)>0\), and hence
\(\Delta_{\R^{2n}}u>0\). If
\(\lambda_1\leq\cdots\leq\lambda_{2n}\) are the eigenvalues of
\(D^2u(z_0)\), the preceding upper bound and
\(\sum_{\alpha=1}^{2n}\lambda_\alpha\geq0\) imply
\[
  \lambda_1
  \geq
  -C\frac{-u(z_0)}{|z_0|^2}.
\]
Consequently,
\[
  |z_0|^2|D^2u(z_0)|\leq C(-u(z_0)).
\]
Finally, using \(-u(z)\leq C_0|z|^\gamma\) with
\(\gamma=2-\frac{2n}{k}\), we obtain
\[
  |D^2u(z)|
  \leq C|z|^{\gamma-2}
  =C|z|^{-2n/k}.
\]
This proves the proposition.
\end{proof}

\section{The Real Hessian Estimate for the Complex Monge-Amp\`ere Equation}\label{sec:cma-hessian}

In the Monge--Amp\`ere case, the argument from the preceding section does not apply directly.
Instead, we exploit the invariance of the equation under volume preserving complex linear
transformations. Accordingly, we set
\[
\mathfrak{sl}(n,\C)
=
\{A\in M_n(\C):\tr A=0\}.
\]

We first record the following elementary linear algebra lemma.
\begin{lemma}\label{lem:sl-linear-algebra}
Assume \(n\geq 2\). For any \(u,v\in\C^n\) with
\(\lvert u\rvert=\lvert v\rvert=1\), there exists
\(B\in\mathfrak{sl}(n,\C)\) such that
\[
Bu=v,\qquad
\norm{B}\leq \sqrt{2},
\qquad
\norm{B^2}\leq 2.
\]
\end{lemma}

\begin{proof}
Suppose first that \(u=e_1=(1,0,\ldots,0)^T\). Define \(B\) by
$$
  B_{j1}=v_j,\quad 1\leq j\leq n,
  \qquad
  B_{22}=-v_1,
$$
with all other entries equal to zero. Then
$$
  Be_1=v,
  \qquad
  \tr B=0,
$$
and hence \(B\in\SL(n,\C)\). Moreover,
$$
  \norm{B}^2
  =\lvert v\rvert^2+\lvert v_1\rvert^2
  \leq 2.
$$
Thus
$
  \norm{B}\leq\sqrt{2}$ and
  $\norm{B^2}\leq\norm{B}^2\leq2.
$

For general \(u\), choose \(U\in U(n)\) such that \(Uu=e_1\), and set
$$
  \widetilde v=Uv.
$$
By the preceding construction, there exists
\(\widetilde B\in\SL(n,\C)\) such that
$$
  \widetilde B e_1=\widetilde v,
  \qquad
  \|{\widetilde B}\|\leq\sqrt{2}.
$$
Let
$
  B=U^{-1}\widetilde B U,
$
then \(B\in\SL(n,\C)\) and
$$
  Bu
  =U^{-1}\widetilde B Uu
  =U^{-1}\widetilde v
  =v.
$$
Since \(U\) is unitary,
$$
  \norm{B}=\|{\tilde B}\|\leq\sqrt{2},
  \qquad
  \norm{B^2}\leq\norm{B}^2\leq2.
$$

\end{proof}

We now establish the global Hessian estimate for the approximating problem
.
Let
\[
  \Omega_\varepsilon=\Omega\setminus \overline{B_\varepsilon(0)}
\]
and let \(u=u_\varepsilon\in C^4(\Omega_\varepsilon)\cap
C^2(\overline{\Omega_\varepsilon})\) be a plurisubharmonic solution of
\begin{equation}\label{eq:cma-problem}
  \det\ (u_{\varepsilon})_{i\bar j}=\varepsilon>0
  \qquad\text{in }\Omega_\varepsilon .
\end{equation}
To prove the pointwise real Hessian estimate, we need  
Guan's crucial boundary Hessian and global gradient estimates in \cite{Guan1998}
\begin{equation}\label{eq:cma-boundary-estimate}
  \sup_{\partial\Omega_\varepsilon}|z|^2|D^2u_\varepsilon(z)|\leq C_2,
\end{equation}
and
\begin{equation}\label{eq:cma-gradient-estimate}
  \sup_{\Omega_\varepsilon}|z|\,|Du_\varepsilon (z)|\leq C_1.
\end{equation}
 Combing the lemma and the above boundary Hessian estimate and the global gradient estimate, we can prove the following real Hessian estimates. 
\begin{proposition}\label{prop:cma-hessian}
There exists a constant \(C>0\), independent of 
\(\varepsilon\), such that
\[
  |D^2u(z)|\leq \frac{C}{|z|^2}
  \qquad\text{in }\overline\Omega_\varepsilon .
\]
\end{proposition}

\begin{proof}
Fix \(z_0\in\Omega_\varepsilon\) and a unit vector
\(\xi=(\xi_1,\ldots,\xi_{2n})\in\R^{2n}\). Set
\(\chi_j=\xi_j+\sqrt{-1}\xi_{j+n}\), \(1\leq j\leq n\). Then
\(|\chi|=1\). By Lemma~\ref{lem:sl-linear-algebra}, there exists a
trace-free \(B\in\SL(n,\C)\), with
\(\norm{B}+\norm{B^2}\leq C\), such that
\[
  B\frac{z_0}{|z_0|}=\chi.
\]
Let \(E_t=e^{tB}\). Since \(\tr B=0\), we have \(\det E_t=1\). For
\(t\) sufficiently close to zero, set \(u_t(z)=u(E_tz)\) whenever
\(E_tz\in\Omega_\varepsilon\). Since \(E_t\) is complex linear,
\[
  \det\bigl((u_t)_{i\bar j}(z)\bigr)
  =|\det E_t|^2
    \det\bigl(u_{i\bar j}(E_tz)\bigr)
  =\varepsilon.
\]

Let \(F=\log\det\). Thus
\(F((u_t)_{i\bar j})=\log \varepsilon\). Differentiating with respect
to \(t\), we obtain
\begin{equation}\label{eq:cma-first-variation}
  F^{i\bar j}\bigl((u_t)_{p\bar q}\bigr)
  \left(\frac{du_t}{dt}\right)_{i\bar j}=0.
\end{equation}
Differentiating once again gives
\begin{equation}\label{eq:cma-second-variation}
  F^{i\bar j}\bigl((u_t)_{p\bar q}\bigr)
  \left(\frac{d^2u_t}{dt^2}\right)_{i\bar j}
  +
  F^{i\bar j,k\bar l}\bigl((u_t)_{p\bar q}\bigr)
  \left(\frac{du_t}{dt}\right)_{i\bar j}
  \left(\frac{du_t}{dt}\right)_{k\bar l}
  =0.
\end{equation}
Since \(F=\log\det\) is concave, the second term in
\eqref{eq:cma-second-variation} is nonpositive. Setting \(t=0\), we obtain
\[
  u^{i\bar j}
  \left(
    \left.\frac{d^2u_t}{dt^2}\right|_{t=0}
  \right)_{i\bar j}
  \geq 0.
\]
Thus, \(\left.\frac{d^2u_t}{dt^2}\right|_{t=0}\) is a subsolution of the
linearized equation.

Next, we compute this second variation explicitly. Identifying
\(\C^n\) with \(\R^{2n}\), and using \(E_t=e^{tB}\), we have
\[
  \frac{d}{dt}(E_tz)=BE_tz,
  \qquad
  \frac{d^2}{dt^2}(E_tz)=B^2E_tz.
\]
Hence
\[
  \frac{du_t(z)}{dt}
  =
  Du(E_tz)\cdot BE_tz.
\]
Differentiating once more with respect to \(t\), we obtain
\[
  \frac{d^2u_t(z)}{dt^2}
  =
  D^2u(E_tz)[BE_tz,BE_tz]
  +Du(E_tz)\cdot B^2E_tz.
\]
Therefore,
\begin{equation}\label{eq:cma-ut-second-derivative}
  \left.\frac{d^2u_t(z)}{dt^2}\right|_{t=0}
  =
  D^2u(z)[Bz,Bz]+Du(z)\cdot B^2z.
\end{equation}

On \(\partial\Omega_\varepsilon\), Lemma~\ref{lem:sl-linear-algebra},
\eqref{eq:cma-boundary-estimate}, and
\eqref{eq:cma-gradient-estimate} give
\[
  \left|
    \left.\frac{d^2u_t(z)}{dt^2}\right|_{t=0}
  \right|
  \leq
  \norm{B}^2|z|^2|D^2u(z)|
  +\norm{B^2}|z|\,|Du(z)|
  \leq C.
\]
The maximum principle therefore yields
\[
  \left.\frac{d^2u_t(z)}{dt^2}\right|_{t=0}\leq C
  \qquad\text{in }\Omega_\varepsilon.
\]

At \(z_0\), the choice of \(B\) gives \(Bz_0=|z_0|\chi\). Under the
identification \(\C^n\simeq\R^{2n}\), the vector \(\chi\) corresponds
to \(\xi\). Hence \eqref{eq:cma-ut-second-derivative} gives
\[
  \left.\frac{d^2u_t(z_0)}{dt^2}\right|_{t=0}
  =
  |z_0|^2D^2u(z_0)[\xi,\xi]
  +Du(z_0)\cdot B^2z_0.
\]
By \eqref{eq:cma-gradient-estimate} and the bound for \(B^2\),
\[
  |Du(z_0)\cdot B^2z_0|
  \leq \norm{B^2}|z_0|\,|Du(z_0)|
  \leq C.
\]
It follows that
\[
  D^2u(z_0)[\xi,\xi]\leq \frac{C}{|z_0|^2}.
\]
Since \(z_0\) and \(\xi\) are arbitrary, every eigenvalue of the real
Hessian \(D^2u(z_0)\) is bounded above by \(C|z_0|^{-2}\).
Therefore, we get
\[
  |D^2u(z_0)|\leq \frac{C}{|z_0|^2}.
\]
Together with \eqref{eq:cma-boundary-estimate}, this proves the proposition.
\end{proof}

\section{The pointwise real Hessian estimate for the exterior problem}\label{section4}

Let \(\Omega\subset\mathbb C^n\) be a smooth strongly pesudoconvex domain such that \(0\in \Omega\) and \(\overline\Omega\)  holomorphically convex in a ball centered at \(0\).
For sufficiently large \(R\), let \(\Sigma_R=B_R(0)\setminus \overline{\Omega}\). Let \(u^{\varepsilon,R}\) be a solution to
\begin{equation}\label{eq::GuanBoMAExterior}
\begin{cases}
\mathrm{det}\ u^{\varepsilon,R}_{i\bar j}=f^{\varepsilon}&\quad \text{in }\Sigma_R,\\
u=0&\quad\text{on }\partial\Sigma_R,\\
u=\log|z|+O(1)&\quad\text{as }|z|\rightarrow\infty.
\end{cases}
\end{equation}
where 
\(f^{\varepsilon}(z)=2^{-n}\varepsilon^2(|z|^2+\varepsilon^2)^{-n-1}\). As proved by Guan in \cite{Guan2007}, there is a positive constant \(C>0\), such that
\[-C\leq u^{\varepsilon,R}-\log|z|\leq C,\quad |Du^{\varepsilon,R}|\leq \frac{C}{|z|},\quad |D^2u^{\varepsilon,R}|\le C\quad  \text{in }\Sigma_R,\]
and 
\[|D^2u^{\varepsilon,R}(z)|<\frac{C}{|z|^2}\quad\text{on }\partial\Sigma_R.\] 
We prove a pointwise decay estimate for real Hessian below.
\begin{proposition}
There is a positive constant independent of \(\varepsilon\) and \(R\) such that
\[|D^2u^{\varepsilon,R}|\leq \frac{C}{|z|^2}\quad \text{in }\overline{\Sigma_R}.\]
\end{proposition}

\begin{proof}
Write \(u=u^{\varepsilon,R}\). Fix \(z_0\in\Sigma_R\) and a unit vector
\(\xi=(\xi_1,\ldots,\xi_{2n})\in\R^{2n}\), and set
\(\chi_j=\xi_j+\sqrt{-1}\xi_{j+n}\). By
Lemma~\ref{lem:sl-linear-algebra}, there exists
\(B\in\mathfrak{sl}(n,\C)\), with
\(\|B\|+\|B^2\|\leq C\), such that
\[
  B\frac{z_0}{|z_0|}=\chi.
\]
Let \(E_t=e^{tB}\) and \(u_t(z)=u(E_tz)\). Since \(\tr B=0\),
\(\det E_t=1\), and hence
\[
  \det\bigl((u_t)_{i\bar j}(z)\bigr)=f^\varepsilon(E_tz).
\]

Set
\[
  w=\left.\frac{d^2u_t}{dt^2}\right|_{t=0},
  \qquad
  L=u^{i\bar j}\partial_i\partial_{\bar j}.
\]
Applying the second-variation argument from the proof of
Proposition~\ref{prop:cma-hessian} to
\[
  \log\det\bigl((u_t)_{i\bar j}\bigr)
  =\log f^\varepsilon(E_tz),
\]
we obtain
\[
  Lw\geq
  \left.\frac{d^2}{dt^2}\log f^\varepsilon(E_tz)\right|_{t=0}.
\]
Since
\[
  \log f^\varepsilon(z)
  =C_\varepsilon-(n+1)\log(|z|^2+\varepsilon^2),
\]
and \(\|B\|+\|B^2\|\leq C\), the right-hand side is bounded below by
\(-C\), uniformly in \(\varepsilon\), \(R\), and \(B\). Thus
\[
  Lw\geq-C.
\]

Define
\[
  \varphi(z)=\sum_{p=1}^n
  \bigl(z_pu_p+\bar z_pu_{\bar p}\bigr).
\]
The gradient estimate gives \(|\varphi|\leq C\). Moreover,
differentiating the equation  yields
\[
\begin{aligned}
  L\varphi
  &=2n+\sum_{p=1}^n
  \left(
    z_p(\log f^\varepsilon)_p
    +\bar z_p(\log f^\varepsilon)_{\bar p}
  \right)\\
  &=-2+2(n+1)\frac{\varepsilon^2}
  {|z|^2+\varepsilon^2}.
\end{aligned}
\]
Since \(0\in\Omega\), there is \(r_0>0\), independent of \(R\), such
that \(|z|\geq r_0\) in \(\Sigma_R\). Hence, for \(\varepsilon\)
sufficiently small,
\[
  L(-\varphi)\geq1
  \qquad\text{in }\Sigma_R.
\]
Choose \(C_0\) sufficiently large and set \(W=w-C_0\varphi\). Then
\(LW\geq0\) in \(\Sigma_R\).

As in the proof of Proposition~\ref{prop:cma-hessian}, let
\[
  w=D^2u[Bz,Bz]+Du\cdot B^2z.
\]
The boundary Hessian estimate and the gradient estimate imply
\(|w|\leq C\) on \(\partial\Sigma_R\). Since \(|\varphi|\leq C\), we also
have \(W\leq C\) on \(\partial\Sigma_R\). The maximum principle gives
\(W\leq C\) in \(\Sigma_R\), and therefore \(w\leq C\).

At \(z_0\), \(Bz_0=|z_0|\chi\), and \(\chi\) corresponds to the real
vector \(\xi\). Hence
\[
  w(z_0)
  =
  |z_0|^2D^2u(z_0)[\xi,\xi]
  +Du(z_0)\cdot B^2z_0.
\]
Using \(|z|\,|Du(z)|\leq C\), we obtain
\[
  D^2u(z_0)[\xi,\xi]\leq\frac{C}{|z_0|^2}.
\]
Since \(z_0\) and \(\xi\) are arbitrary, all eigenvalues of the real
Hessian are bounded above by \(C|z|^{-2}\). Finally, \(u\) is
plurisubharmonic, so \(\Delta u\geq0\); the trace condition gives the
corresponding lower bound. Therefore
\[
  |D^2u(z)|\leq\frac{C}{|z|^2}
  \qquad\text{in }\overline{\Sigma_R}.
\]
\end{proof}

Let \(\Omega\subset\mathbb C^n\) be a smoothly strongly pesudoconvex domain such that \(0\in \Omega\) and \(\overline\Omega\) is holomorphically convex in a ball centered at \(0\).
For sufficient large \(R\), let \(\Sigma_R=B_R(0)\setminus \overline{\Omega}\). Let \(u^{\varepsilon,R}\) be a solution to
\begin{equation}\label{eq::GMZkExterior}
\begin{cases}
\sigma_k(u^{\varepsilon,R}_{i\bar j})=f^{\varepsilon}&\quad \text{in }\Sigma_R,\\
u=-1&\quad\text{on }\partial\Sigma_R,\\
u\rightarrow0&\quad\text{as }|z|\rightarrow\infty.
\end{cases}
\end{equation}
where 
\(f^{\varepsilon}(z)=C_{n,k}\varepsilon^2(1+\varepsilon^2)^{n-k}(|z|^2+\varepsilon^2)^{-n-1}\), \(C_\varepsilon>0\) is a constant. As proved in \cite{GaoMaZhangExterior}, there is a positive constant \(C>0\), such that
\begin{equation}\label{est::boundarykexterior1}
    C^{-1}|z|^{2-\frac{2n}k}<-u^{\varepsilon,R}<C|z|^{2-\frac{2n}k} ,\quad |Du^{\varepsilon,R}(z)|<C|z|^{1-\frac{2n}k}\quad \text{in }\Sigma_R,
\end{equation}
and 
\begin{equation}\label{est::boundarykexterior2}
    |D^2u^{\varepsilon,R}(z)|<C|z|^{-\frac{2n}k}\quad\text{on }\partial\Sigma_R.
\end{equation} 
We prove a pointwise dacay estimate for real Hessian in the following proposition.
\begin{proposition}
There is a positive constant independent of \(\varepsilon\) and \(R\) such that
\[|D^2u^{\varepsilon,R}|\leq C|z|^{-\frac{2n}{k}}.\]
\end{proposition}

\begin{proof}
Denote by \(u=u^{\varepsilon,R}\) for short. Let \(z_0\in \Sigma_R\) be a fix point. For any unit vector \(\xi=(\xi^1,\cdots,\xi^{2n})\in\mathbb R^{2n}\), by lemma \ref{lem:linear-algebra}, there exist \(a\in\mathbb R\) and \(K\in \mathfrak u(n)\), such that \(B:=aI+K\) satisfies
\[B\frac{z_0}{|z_0|}=\chi,\quad |a|+\|B\|+\|B^2\|\leq C_n,\] where \(\chi_j=\xi_j+\sqrt{-1}\xi_{j+n}\) and \(C_n\) is a positive constant depending only on \(n\).

As in the proof of proposition \ref{prop:global-hessian}, set \(E_t:=e^{tB}\), which is a small pertubation of the identity matrix. Since \(B=aI+K\) , \(E_t=e^{at}U_t\), where \(U_t:=e^{tK}\). Since \(K^*=-K\),  \(U_t^*U_t=e^{tK^*+K}=I\). So \(U_t\in U(n)\). For any \(z\in\Sigma_R\), and \(t\) sufficent small,  define 
\[u_t(z)=u(E_tz).\]
Then 
\[D^2_{\mathbb C}u_t=e^{2at}U_{t}^TD^2_{\mathbb C}u(E_tz)\overline U_t.\]
By unitary invariance of \(\sigma\), it follows that
\[\sigma_{k}(D^2_{\mathbb C}u_t)=e^{2akt}\sigma_k(D^2_{\mathbb C}u(E_tz)).\] 
Let \(F=\log\sigma_k\), and \(\tilde f(z)=\log C_{n,k}+2\log\varepsilon+(n-k)\log(1+\varepsilon^2)-(n+1)\log(|z|^2+\varepsilon^2)\). Then \(F(D^2_{\mathbb C}u_t(z))=2akt+\tilde f(E_tz)\), with \(\tilde f(E_tz)=\log C_n^k+2\log\varepsilon+(n-k)\log(1+\varepsilon^2)-(n+1)\log(e^{2at}|z|^2+\varepsilon^2)\).
Differentiating with respect \(t\) gives
\[F^{i\bar j}(D^2_{\mathbb C}u_t)(\frac{\partial u_t}{\partial t})_{i\bar j}=2ak+\frac{\partial \tilde f(E_tz)}{\partial t}.\]
Since
\[\frac{d}{dt}E_t=\frac{d}{dt}e^{tB}=Be^{tB}=BE_t,\]
we have
\[\frac{\partial}{\partial t}u_t(z)=\frac{\partial}{\partial t}u(E_tz)=Du(E_tz)\cdot\frac{\partial }{\partial t}E_tz=Du(E_tz)\cdot BE_tz.\]
and
\[\begin{aligned}
    \frac{\partial^2}{\partial t^2}u_t=&Du(E_tz)\cdot B^2E_tz+D^2_{\mathbb C}u(E_tz)[BE_tz,BE_tz].
\end{aligned}\]
Differentiating once more, we obtain
\[F^{i\bar j}(D^2_{\mathbb C}u_t)(\frac{\partial^2u_t}{\partial t^2})_{i\bar j}+F^{i\bar j,r\bar s}(\frac{\partial u_t}{\partial t})_{i\bar j}(\frac{\partial u_t}{\partial t})_{r\bar s}=\frac{\partial^2f(E_tz)}{\partial t^2}.\]
Since \(F\) is concave in \(\Gamma_k\), the second term in above equality is nonpositve. Hence
\[L\frac{\partial^2 u_t}{\partial t^2}\geq \frac{\partial^2f(E_tz)}{\partial t^2}=\frac{-4a^2e^{2at}(n+1)|z|^2}{e^{2at}|z|^2+\varepsilon^2}+\frac{4a^2e^{2at}|z|^4}{(e^{2at}|z|^2+\varepsilon^2)^2},\]
where 
\(L=F^{i\bar j}(D^2_{\mathbb C})\frac{\partial^2}{\partial z_i\partial\bar z_j}\). Let \(t=0\), we have
\[\begin{aligned}
    \frac{\partial^2}{\partial t^2}u_t\bigg|_{t=0}=&Du\cdot B^2z+D^2u[Bz,Bz].
\end{aligned}\]
\[\mathrm{LHS}=L\bigg(\frac{\partial^2}{\partial t^2}u_t\bigg|_{t=0}\bigg),\]
and
\[\begin{aligned}
\mathrm{RHS}=&\frac{-4(n+1)a^2|z|^2\varepsilon^2}{(|z|^2+\varepsilon^2)^2}.
\end{aligned}\]
Then
\[L\bigg(\frac{\partial^2}{\partial t^2}u_t\bigg|_{t=0}\bigg)\geq -C_n.\]
On the other hand, since
\[Lu=k,\]
for a sufficently large \(M\) depending only on \(n\) and \(k\), there holds
\[L\bigg(\frac{\partial^2}{\partial t^2}u_t\bigg|_{t=0}+Mu\bigg)\geq 0\quad \text{in }\Sigma_R.\]
By \eqref{est::boundarykexterior1} and \eqref{est::boundarykexterior2}, \
\[\frac{\partial^2}{\partial t^2}u_t\bigg|_{t=0}+Mu\leq 0\quad\text{on }\partial\Sigma_R,\]
it follows from the maximum principle that
\[\frac{\partial^2}{\partial t^2}u_t\bigg|_{t=0}+Mu\leq 0\quad\text{in }\Sigma_R.\]
Hence \[\frac{\partial^2}{\partial t^2}u_t(z_0)\bigg|_{t=0}\leq C|z_0|^{2-\frac{2n}k},\] where \(C\) in above inequality is independent of \(R\), \(\varepsilon\), \(z_0\) and \(\xi\).

Since 
\[|z_0|^2\xi^TD^2u(z_0)\xi=\frac{\partial^2}{\partial t^2}u_t(z_0)\bigg|_{t=0}-\langle Du(z_0),B^2z_0\rangle.\]
We get
\[\xi^TD^2u(z_0)\xi\leq C|z_0|^{-\frac{2n}k}.\]
Since \(C\) in above estimate is independent of \(z_0\) and \(\xi\), and \(u\) is subharmonic, we finally obtain 
\[|D^2u|\leq C|z|^{-\frac{2n}k}\quad \text{in }\Sigma_R.\]

\end{proof}

\bigskip
\begingroup
\small
\noindent\textsuperscript{1}School of Mathematics and Statistics, Xi'an Jiaotong University,
Xi'an 710049, Shaanxi Province, China

\noindent\textit{E-mail address}: \href{mailto:gzhmath@xjtu.edu.cn}{\nolinkurl{gzhmath@xjtu.edu.cn}} (Zhenghuan Gao)

\medskip
\noindent\textsuperscript{2}School of Mathematical Sciences, University of Science and Technology of China,
Hefei 230026, Anhui Province, China

\noindent\textit{E-mail address}: \href{mailto:xinan@ustc.edu.cn}{\nolinkurl{xinan@ustc.edu.cn}} (Xi-Nan Ma)

\noindent\textit{E-mail address}: \href{mailto:baoyu1@mail.ustc.edu.cn}{\nolinkurl{baoyu1@mail.ustc.edu.cn}} (Bao Yu)

\medskip
\noindent\textsuperscript{3}School of Mathematical Sciences, East China Normal University,
Shanghai 200241, China

\noindent\textit{E-mail address}: \href{mailto:dkzhang@math.ecnu.edu.cn}{\nolinkurl{dkzhang@math.ecnu.edu.cn}} (Dekai Zhang)
\endgroup


\begin{thebibliography}{99}

\bibitem{Blocki2000}
Z. B\l{}ocki,
\emph{The \(C^{1,1}\) regularity of the pluricomplex Green function},
Michigan Math. J. \textbf{47} (2000), no. 2, 211--215.

\bibitem{Blocki2005}
Z. B\l{}ocki,
\emph{Weak solutions to the complex Hessian equation},
Ann. Inst. Fourier (Grenoble) \textbf{55} (2005), no. 5, 1735--1756.

\bibitem{CNS1986}
L. A. Caffarelli, L. Nirenberg, and J. Spruck,
\emph{The Dirichlet problem for the degenerate Monge-Amp\`ere equation},
Rev. Mat. Iberoam. \textbf{2} (1986), no. 1, 19--27.

\bibitem{GaoMaZhang2023}
Z. Gao, X.-N. Ma, and D. Zhang,
\emph{The Dirichlet problem of homogeneous complex \(k\)-Hessian equation
in a \((k-1)\)-pseudoconvex domain with isolated singularity},
arXiv:2304.08407 (2023).

\bibitem{GaoMaZhangReal2023}
Z. Gao, X.-N. Ma, and D. Zhang,
\emph{The Dirichlet problem of the homogeneous \(k\)-Hessian equation
in a punctured domain}, arXiv:2303.07976 (2023).

\bibitem{GaoMaZhangExterior}
Z. Gao, X.-N. Ma, and D. Zhang,
\emph{The exterior Dirichlet problem for the homogeneous complex
\(k\)-Hessian equation},
Adv. Nonlinear Stud. \textbf{23} (2023), no. 1,
doi:10.1515/ans-2022-0039.

\bibitem{Guan1998}
B. Guan,
\emph{The Dirichlet problem for complex Monge-Amp\`ere equations and regularity of the pluricomplex Green function},
Comm. Anal. Geom. \textbf{6} (1998), 687--703.


\bibitem{Guan2007}
B. Guan,
\emph{On the regularity of the pluricomplex Green functions},
Int. Math. Res. Not. IMRN (2007), Art. ID rnm106, 19 pp.,
doi:10.1093/imrn/rnm106.


\bibitem{Krylov1995}
N. V. Krylov,
\emph{A theorem on the degenerate elliptic Bellman equations in bounded domains},
Differential Integral Equations \textbf{8} (1995), 961--980.



\bibitem{Li2004}
S.-Y. Li,
\emph{On the Dirichlet problems for symmetric function equations of the
eigenvalues of the complex Hessian},
Asian J. Math. \textbf{8} (2004), no. 1, 87--106.
\end{thebibliography}
\end{document}